\documentclass[draft]{amsart}
\usepackage{amssymb,latexsym,amscd, amsthm, epsfig,amsmath,amssymb}
\usepackage{color}
\usepackage{array}
\usepackage{enumerate}
\usepackage{enumitem}
\usepackage{refcount}
\usepackage{algorithm}
\usepackage{algorithmic}
\usepackage{boxedminipage}
\usepackage{float}
\usepackage[utf8x]{inputenc}
\newcolumntype{C}{>{$}c<{$}}
\newcolumntype{L}{>{$}l<{$}}
\newcolumntype{R}{>{$}r<{$}}

 \newtheorem{thm}{Theorem}[section]

 \newtheorem{prop}[thm]{Proposition}
 \theoremstyle{definition}
 \newtheorem{defn}[thm]{Definition}
 \theoremstyle{remark}
 \newtheorem{rem}[thm]{Remark}
 \theoremstyle{definition}
 \newtheorem{ex}[thm]{Example}

\newtheorem{prob*}{Problem}

 \DeclareMathOperator{\Supp}{Supp}
 \DeclareMathOperator{\Sing}{Sing}

 \newcommand{\PP}{\mathbb{P}}

\def\move-in{\parshape=1.75true in 5true in}

\def\a{\bigskip \par \noindent}
\def\b{\par \noindent}
\def\dd{\medskip \par \noindent}
\long\def\eatit#1{}

\def\C{\Bbb C}

\newcommand{\prfend}{\hbox to7pt{\hfil}
\par\vskip-\baselineskip\hbox to\hsize
{\hfil\vbox {\hrule width6pt height6pt}}\vskip\baselineskip}

\long\def\eatit#1{}

\def\C{\Bbb C}

\def\l{\lambda}

\def\e{\varepsilon}

\usepackage{xypic}

\begin{document}

\title{Remarks on double points of plane curves}

\author[A. Gimigliano]{Alessandro Gimigliano}
\address[Alessandro Gimigliano]{Dipartimento di Matematica, Università di Bologna, Italy}
\email{alessandr.gimigliano@unibo.it }

\author[M. Id\`a]{Monica Id\`a}
\address[Monica Id\`a]{Dipartimento di Matematica, Università di Bologna, Italy}
\email{monica.ida@unibo.it }

\maketitle


\begin{abstract}
We study the relation between the type of a double point of a plane curve and the curvilinear 0-dimensional subschemes of the curve at the point. An Algorithm related to a classical procedure for the study of double points via osculating curves is described and proved. Eventually we look for a way to create examples of rational plane curves with given singularities $A_s$.
 \end{abstract}


\section{Introduction}

This note is dedicated to the study of double points of plane curves, either using their implicit equation or, in the case of rational curves, their parameterization. This is quite a classical subject in Algebraic Geometry; the aim of the present paper is to study the structure of a double point of a plane curve via the curvilinear 0-dimensional subschemes of the curve at the point, and to give, in the case of plane curves defined by their implicit equation, an algorithm which, following a classical procedure, allows to describe the structure of a double point.
We work on the complex field.
\par \medskip We recall that a singularity of type $A_s$ for a plane curve is a double point that can be resolved via $r$ blow ups if $s=2r-\e$, $\e =0,1$ and the desingularization yields two points if $\e =1$ and only one if $\e =0$. In the following, given a curve $D$ smooth at a point $Q$, with $mQ$ we denote the curvilinear 0-dimensional scheme of length $m$ supported at $Q$ and contained in $D$.

\medskip If $C \subset \PP^2$ is a degree $n$ integral rational curve, to give a parameterization means to see $C$ as the projection of a rational normal curve $C_n\subset \PP^n$ (see section 3 for details). In Theorem 4.3 of \cite{BGI} we show that the point $P$ of $C$ is a double point of type $A_{2r-\e}$ if and only if there is a 0-dimensional curvilinear scheme $X\subset C$ of length $r$, projection of a certain curvilinear scheme $Y$ of length $2r$ on $C_n$ (supported on two points if $\e =1$, on one point if $\e=0$) and $X$ is ``maximal" with respect to this property.
 In section 2 of the present paper we generalize the result to any integral plane curve (Theorem \ref{2mtom notrat}). In the same section we study the intersection multiplicities of a plane curve $C$ with a double point $P$ with curves smooth at $P$ (Proposition \ref{Lemma}; this will be useful in the last section) and we give another characterization of $A_s$ singularities through the possible length of the curvilinear schemes supported at $P$ and contained in $C$.

\medskip In section 3 we describe an algorithm which classifies double points on any plane curve; this algorithm is based on the classical method of studying the osculating curves (parabolas of degree $r$) to a curve at a double point. We give this algorithm in detail and also its justification since, although classically known, we find its main references (e.g. see \cite{C}) a bit cumbersome in justifying the procedure.

\medskip In section 4 we build an example illustrating the previous algorithm.

\medskip In section 5 we give a brief summary of the techniques, introduced in \cite {BGI}, useful for studying singular points on a rational plane curve when the parameterization is known.

\medskip In section 6, with \cite{CoCoA}, we build an example which illustrates how to build a plane rational curve with a double point of a chosen type using projection techniques, and how to use the techniques of section 5 to study the singularities of a rational curve, given its parameterization.

\a
\section{Double points and curvilinear schemes}
The following theorem is a generalization of \cite{BGI},Theorem 4.3; its proof is essentially the same as in the rational case:

\begin{thm}\label{2mtom notrat}   Let $ C \subset \mathbb{P}^2$ be a curve with normalization $\pi: \tilde C \to C$ and let $P$ be a double point for $C$. Then:
 \begin{enumerate}
\item\label{a} $P$ is an $A_{2m-1}$ singularity if and only if
\begin{enumerate}
 \item\label{ai}  $\pi^{-1}(P) = \{Q_1, Q_2\} $,
 \item\label{aii}  $\pi (mQ_1|_{\tilde C}\cup mQ_2|_{\tilde C})= X $, where $X$ is a curvilinear scheme of length $m$ contained in $C$,
 \item\label{aiii}  $m$ is the maximum integer for which \eqref{aii} holds.
\end{enumerate}

\item\label{b}  $P$ is an $A_{2m}$ singularity if and only if
\begin{enumerate}
\item\label{bi} $\pi^{-1}(P) = \{Q \}$,
\item\label{bii} $\pi (2mQ|_{\tilde C}) = X $, where $X$ is a curvilinear scheme of length $m$ contained in $C$,
\item\label{biii}  $m$ is the maximum integer for which \eqref{bii} holds.
\end{enumerate}

\end{enumerate}
\end{thm}

\begin{proof} 

It is known that a {\it normal form} for an $A_s$ singularity is given by the curve $\Gamma_s \subset \PP^2$ of equation: $y^2z^{s-1}-x^{s+1}=0$ at the point $P = [0,0,1]$ (e.g. see \cite{refHa, refKP}), i.e. if $C$ has an $A_s$ singularity at $P$, then it is  analytically isomorphic to $\Gamma_s$ at $P$, and has the same multiplicity sequence as $\Gamma_s$ at $P$. Hence we will work on these curves first. 

\dd Case \eqref{a}: Let $(C,P)$ be an $A_{2m-1}$ singularity. Hence $\Gamma_{2m-1}$, in affine coordinates, is the union of the two smooth curves $\Gamma$: $\{y-x^m=0\}$ and $\Gamma '$: $\{y+x^m=0\}$, while $p=(0,0)$. The normalization of $\Gamma \cup \Gamma '$ is  $\psi :\tilde{\Gamma} \cup \tilde{\Gamma '}\rightarrow \Gamma \cup \Gamma '$, where $\Gamma \cup \Gamma '$ is the union of two disjoint lines and the inverse image of $p$ is given by two points, $q\in \tilde{\Gamma}$ and  $q'\in \tilde{\Gamma '}$ , so $\eqref{ai}$ is true. Since $\psi |_{\tilde{\Gamma}} \to \Gamma $ and $\psi |_{\tilde{\Gamma'}} \to \Gamma' $ are two isomorphisms, it is clear that  $\psi(mq|_{\tilde{\Gamma}}) = mp|_{\Gamma}$ and $\psi(mq'|_{\tilde{\Gamma '}}) = mp|_{\Gamma '}$.

Now let's notice that $\Gamma \cap \Gamma'$ is the curvilinear scheme $Z$ whose ideal is $(y,x^m)$, hence $mp|_{\Gamma} = mp|_{\Gamma'} = Z$, and $\eqref{aii}$ is true.  On the other hand, if we consider $\psi((m+1)q|_{\tilde{\Gamma}} \cup (m+1)q'|_{\tilde{\Gamma '}})$, we get the scheme $(m+1)p|_{\Gamma} \cup (m+1)p|_{\Gamma '}$ which corresponds to the ideal $(x^{m+1},xy,y^2)$, and this scheme is not curvilinear, so \eqref{aiii} holds.

 \dd Case \eqref{b}: Let $(C,P)$ be an $A_{2m}$ singularity (hence \eqref{bi} holds). We have that $\Gamma_{2m}$, in affine coordinates, is the irreducible curve $y^2-x^{2m+1}=0$, hence a parameterization (which is also a normalization) for it is $\phi : \mathbb{A}^1 \to \Gamma_{2m}$, where $\phi(t) = (t^2,t^{2m+1})$ and $p=(0,0)$ is such that $\phi^{-1}(p)=q$.  The ring map corresponding to $\phi$ is:  
$$
\tilde{\phi} : \frac{K[x,y]}{(y^2-x^{2m+1})} \to K[t], \quad \overline{x} \mapsto t^2, \quad \overline{y} \mapsto t^{2m+1}.
$$ 

The scheme $2mq|_{\mathbb{A}^1}$ corresponds to the ideal $(t^{2m})$, and $\phi^{-1} ((t^{2m}))= (y,x^m) $, hence \eqref{bii} is true. 

On the other hand, the scheme $2(m+1)q|_{\mathbb{A}^1}$ corresponds to the ideal $(t^{2m+2})$, and $\phi^{-1} ((t^{2m+2}))= (x^{m+1},xy,y^2) $, hence \eqref{bii} is true. 

\dd To check that the ``~if~'' part of statements \eqref{a} and \eqref{b} holds, just consider that \eqref{ai}, respectively \eqref{bi}, determines if the singularity is of type $A_{2h}$, respectively $A_{2h-1}$, while \eqref{aii} and \eqref{aiii}, respectively \eqref{bi} and \eqref{biii}, force $h$ to be equal to $m$. 

Now let us notice that, being $\Gamma_s$  at $p$ analytically isomorphic to $C$ at $P$, when we consider $\Gamma_s$ and $C$ as analytic complex spaces, there exist open euclidean neighborhoods $U$ of $p$ and $V$ of $P$ such that $U\cap \Gamma_s$ and $V\cap C$ are biolomorphically equivalent. Since the statement is of local nature, this is enough to conclude.
 \end{proof}
 
\a We denote the intersection multiplicity of two curves $C$ and $D$ at a point $P$ by $i(C,D,P)$. The Proposition below relates the type $A_s$ of a double point $P\in C$ with the value of $i(C,D,P)$ for a curve $D$ smooth at $P$.

\begin{prop}\label{Lemma}  Let $C$ be a plane reduced curve and $P\in C$ a double point of $C$. Let $D$ be a plane curve, smooth at $P$. Then:

\dd i) Assume $P$ is an $A_{2r-1}$ or an $A_{2r}$ singularity. If $i(C,D,P) \leq 2r$ , then it is an even number.

\dd ii) If $P$ is an $A_{2r-1}$ singularity for $C$, there are curves $D_1$ and $D_2$ smooth at $P$ such that $i(C,D_j,P)\geq 2r+1$ for $j=1,2$ and $i(D_1,D_2,P)=r$.

\dd iii)  If $P$ is an $A_{2r}$ singularity for $C$, then $i(C,D,P)\leq 2r+1$, and there exist curves smooth at $P$ which attain the equality. If $D_1, D_2$ are curves such that $i(C,D_j,P)= 2r+1$ for $j=1,2$, then $i(D_1,D_2,P)>r$.

\dd iv)  Let $O$ be an $A_{s}$ singularity for $C$, with $s=2r-1$ or $s=2r$, $s\geq 2$, and suppose that the tangent of $C$ at $O$ is not the $y$-axis. Then any  curve  $D$, smooth at $O$ and such that $i(C,D,O)\geq 2r+1$,  has a local analytic equation of the form  $$y=\sum_{i= 2}^{r-1} c_ix^i +c_rx^r+\sum_{i\geq r+1} c_ix^i $$ where   $\sum_{i\geq r+1} c_ix^i $ is convergent, and $c_2,...,c_{r}$ are fixed if $s=2r$, while $c_2,...,c_{r-1}$ are fixed and there are only two (different) possibilities for  $c_r$ if  $s=2r-1$.
\end{prop}

\begin{proof}    The curve $C$ at $P$ is analytically isomorphic, at $O$,  to the curve $y^2-x^s=0$, where $s=2r$ if $P$ is an $A_{2r-1}$ singularity and $s=2r+1$ if $P$ is an $A_{2r}$ singularity. Since the intersection multiplicity of two curves is an analytic invariant, from now on we study the multiplicity intersection at $O$ of each of these curves with a curve $D$ smooth at $O$.
\a {\it i)} If $C$ and $D$ meet transversally, $i(C,D,O)=2$ so we are done. 
\b If the tangent of $D$ at $O$ is $y=0$, by the analytic implicit function Theorem (see for example \cite{H}, theorem 2.1.2 p.24), the curve $D$ is locally given by an analytic equation  $y=\sum_{i\geq 2} c_ix^i$.
\b  Denoting by $\omega(S)$ the order of a series $S$, the intersection multiplicity  $i(C,D,O)$ is $$\omega ((\sum_{i\geq 2}c_ix^i)^2-x^s) \quad \eqno{(*)}$$ 
\b Since $\omega ((\sum c_ix^i)^2)= 2\omega (\sum c_ix^i)$ is always even, if $i(C,D,O)\leq 2r$ we have that $\omega ((\sum_{i\geq 2}c_ix^i)^2-x^s)$ is even in both cases $s=2r, 2r+1$, hence statement $ 1)$ is proved.
We have
$$(\sum_{i\leq 2} c_ix^i)^2 = \sum_{k\geq 4} \alpha_k x^k {\quad \rm where \quad} \alpha_k =  \left\{ \begin{array}{l}
2\sum _{j=2}^{{k\over2}-1}c_jc_{k-j}+c_{k\over2}^2 \quad {\rm if \;} k {\; \rm even \;}\\
\\
2\sum _{j=2}^{{k-1\over2}}c_jc_{k-j} \quad {\rm if \;} k {\; \rm odd \;}
\end{array}\right.$$

\dd Since $\omega (\sum_{k\geq 4} \alpha_k x^k) $ is always even, we have $\alpha_4=\dots=\alpha_{2m}=0 \Rightarrow \alpha_{2m+1}=0$, and $$\alpha_4=\dots=\alpha_{2m}=\alpha_{2m+1}=0 \quad \iff \quad c_2=\dots=c_m=0 \eqno{(\dagger)}$$

\a {\it ii)}  Let $O$ be an $A_{2r-1}$ singularity, i.e. $C: y^2-x^{2r}=0$, and assume $i(C,D,O) \geq 2r+1$. Since $i(C,D,O)= \omega(\sum_{k\geq 4} \alpha_k x^k -x^{2r} )$, we must have $\alpha_4=\dots=\alpha_{2r-2}=0$, this implying $ \alpha_{2r-1}=0$, and $\alpha_{2r}=1$, hence $c_2=\dots=c_{r-1}=0$ and $c_r^2=1$. Hence  there are two families of curves smooth at $O$ and with intersection multiplicity $\geq  2r+1$ with $C$ at $O$, namely, those with a local equation of the form $y=x^r +\sum _{i\geq r+1}c_ix^i$ or $y=-x^r +\sum _{i\geq r+1}c_ix^i$. If $D_1$ is a curve of the first family and $D_2$ of the second, we have $i(D_1,D_2,O)=r$.

\a {\it iii)}  Let $O$ be an $A_{2r}$ singularity, i.e. $C: y^2-x^{2r+1}=0$, and assume $i(C,D,O) \geq 2r+1$. Since $i(C,D,O)= \omega(\sum_{k\geq 4} \alpha_k x^k -x^{2r+1} )$, 
we must have $\alpha_4=\dots=\alpha_{2r}=0$ and this implies  $ \alpha_{2r+1}=0$, hence the coefficient of $x^{2r+1}$ in the series $\sum_{k\geq 4} \alpha_k x^k -x^{2r+1}$ is always $-1$, so that the order of the series has to be $2r+1$; in other words,  $i(C,D,O)= 2r+1$. In this case $(\dagger)$ says that 
$D$ has a local equation of the form $y=\sum _{i\geq r+1}c_ix^i$. 
Hence, if $D_1, D_2$ are two such curves, we find $i(D_1,D_2,O)>r$.

\a {\it iv)}  The statement follows immediately from what is written above when $C$ has equation $y^2-x^s=0$. The general case is attained taking a local biholomorphy sending the curve to $y^2-x^s=0$.   

\end{proof}

\medskip
The following Theorem \ref{curvi}  gives a description of a double point $P$ of a plane curve in terms of the curvilinear 0-dimensional subschemes of the curve supported at $P$.

\medskip
\begin{thm}\label{curvi}  Let $C$ be a plane reduced curve and $P\in C$ a double point of $C$. Then $P$ is an $A_{2r}$ singularity if and only if no curvilinear scheme supported at $P$ of length $> 2r+1$ is contained in $C$. More precisely,

\dd i)  $P$ is an $A_{2r-1}$ singularity for $C$ if and only if  for any $\ell \geq 1$ there is a curvilinear scheme supported at $P$ of length $\ell$ contained in $C$;

\dd ii) $P$ is an $A_{2r}$ singularity if and only if for any $\ell \leq 2r+1$ there is a curvilinear scheme supported at $P$ of length $\ell $ contained in $C$, and no curvilinear scheme supported at $P$  of length $> 2r+1$ contained in $C$.
\end{thm}
\begin{proof}

The curve $C$ at $P$ is analytically isomorphic  to the curve $\Gamma_s: y^2-x^s=0$ at $O$, where $s=2r$ if $P$ is an $A_{2r-1}$ singularity and $s=2r+1$ if $P$ is an $A_{2r}$ singularity. 
\dd {\it i)}  If $P$ is an $A_{2r-1}$ singularity, let $\ell \geq 1$, and consider the 0-dimensional curvilinear scheme $Z$ of ideal $(y-x^r,x^\ell)$ supported at $O$; $Z$ has length $\ell$ and is contained in $\Gamma_{2r}$, since $y^2-x^{2r}\in (y-x^r,x^\ell)$.

\dd {\it ii)}  If $P$ is an $A_{2r}$ singularity, let $1\leq \ell \leq 2r+1$, and consider the 0-dimensional curvilinear scheme $Z$ of ideal $(y,x^\ell)$ supported at $O$; $Z$ has length $\ell$ and is contained in $\Gamma_{2r+1}$, since $y^2-x^{2r+1}\in (y,x^\ell)$.
\dd Now assume that a 0-dimensional curvilinear scheme $Y$ of length $ h \geq 2r+2$ is contained in $\Gamma_{2r+1}$; $Y$ being curvilinear, there is a curve $D: g(x,y)=0$ smooth at $O$ and such that $Y\subset D$, so that $I_Y=(g)+(x,y)^h$. Since $Y\subset \Gamma_{2r+1}$, we have $y^2-x^{2r+1} \in I_Y$, hence $(y^2-x^{2r+1}, g)+ (x,y)^h=(g)+(x,y)^h$; 
by \ref{Lemma} $ i(\Gamma_{2r+1},D,O) \leq 2r+1$, so that:
$$2r+1 \geq  i(\Gamma_{2r+1},D,O) = {\rm dim} \left( \C[x,y]/(y^2-x^{2r+1}, g)\right)_{(x,y)}\geq $$ $$\geq  {\rm dim} \left( \C[x,y]/(y^2-x^{2r+1}, g)+(x,y)^h\right)_{(x,y)}= {\rm dim} \left( \C[x,y]/( g)+(x,y)^h\right)_{(x,y)}= h$$
Hence a 0-dimensional curvilinear scheme of length $ \geq 2r+2$ cannot be contained in $\Gamma_{2r+1}$.

\end{proof}

\begin{rem} \rm Notice that not all the 0-dimensional subschemes of $C$ appearing in the statement of \ref{curvi} are cut on $C$ by a curve smooth at $P$, for example by \ref{Lemma} if the length is $\leq 2r$ then any such curve has an even intersection mutiplicity with $C$, while clearly there are subschemes of odd length.
\end{rem}

\a 
\section{An algorithm for determining the type of a double point via implicit equation}

What we will expose here is an algorithm to determine the nature of a double point on a plane curve $C$, given the implicit equation of $C$. The matter is classically known, but we prefer to give it here since it can be a bit "forgotten", especially for younger mathematicians, and also because our main reference (\cite{C}) is a bit cumbersome when giving the justification of this procedure: for this reason we prefer to describe it in a more algorithmic way and to give a rather simple justification of it. 

\medskip

We work with a reduced algebraic curve $C\subset \PP^2=\PP^2_{\mathbb C}$ of degree $n$, given by a homogeneous polynomial  $F\in {\mathbb C}[x_0,x_1,x_2]_n$, and we suppose that the point $O = [0:0:1]$ is a double point for $C$.  Let
$$
  F = \sum _{i+j+k=n}a_{ij}x_0^ix_1^jx_2^{k} .
$$
where each $a_{ij}\in {\mathbb C}$, $i,j,k\in \{0,...,n\}$ and $k=n-i-j$.
\medskip
Since what we want is to study the curve at $O$, we can work in the affine chart $\{x_2\neq 0\}$, with affine coordinates $x =\frac {x_0}{x_2}$, $y=\frac{x_1}{x_2}$, so $O=(0,0)$ and the affine curve is defined by the polynomial:
$$f(x,y) = \sum a_{ij}x^iy^j  $$
The point $O$ being a double point for $C$, we have  $a_{00} = a_{10}=a_{01}= 0$ and $(a_{20}, a_{11}, a_{02}) \neq (0,0,0)$:
$$
f(x,y) = a_{20}x^2 + a_{11}xy + a_{02}y^2 + a_{30}x^3 + a_{21}x^2y + a_{12}xy^2 + a_{03}y^3 + ...     .
$$
We make use of auxiliary curves $\Gamma ^t_{\Lambda}$ given by the equations :
$$
\Gamma ^t_{{\Lambda}}:\   y={\l_1}x + {\l_2}x^2 + ... + {\l_t}x^t 
$$
where $t\geq 1$, ${\Lambda}:= (\l_1,...,\l_t) \in {\mathbb C}^t$ . We look among them for curves osculating the curve $C$ at $O$; notice that the curve $\Gamma ^t_{{\Lambda}}$ is smooth at $(0,0)$ and its degree is $\leq t$ (some of the $\l_i$'s can be zero).

\dd We denote with  $i(C,D,P)$ the intersection multiplicity of two curves $C$,$D$ at the point $P$, and we set $$R(C,\Gamma ^t_{\Lambda}):= f(x,\l_1x + \l_2x^2 + ... + \l_tx^t) \in {\mathbb C}[x]$$ so that $i(C,\Gamma ^t_{\Lambda},O)$ is the least degree assumed by $x$ in $R(C,\Gamma ^t_{\Lambda})$.

\dd The aim of the procedure is to establish the type $A_s$  of the double point $O$. Here we illustrate how to get this result in several steps, before we give a formal algorithm to do that:

\bigskip
\noindent {\it Step 1:} {\bf Analysis via lines $\Gamma^1_\Lambda$.} Since $\; \Gamma ^1_{{\Lambda}}:\   y={\l_1}x\;$ we have
$$
R(C,\Gamma ^1_{\Lambda}) = (a_{20} + a_{11}\l_1 + a_{02}\l_1^2)x^2 + (a_{30} + ... + a_{03}\l_1^3)x^3 + ...
$$
The coefficient of $x^2$ is zero when $$a_{20} + a_{11}\l_1 + a_{02}\l_1^2=0 \eqno{(\star)}$$
There are two cases:

\dd {\it Step 1a:}  If $a_{11}^2-4a_{20}a_{02}\neq 0$, then $(\star)$ has two distinct roots $\l_{11}, \l_{12}$, so there are exactly two tangent lines  $\Gamma ^1_{\l_{11}}$, $\Gamma ^1_{\l_{12}}$ for which   $i(C,\Gamma ^1_{\l_{1j}},O) \geq 3$ ($j=1,2$), \underline{$O$ is a double point} \underline{ $A_1$ } (an ordinary node) for $C$ and the analysis ends here. 

\dd {\it Step 1b:}  If $a_{11}^2-4a_{20}a_{02}= 0$, then $(\star)$ has one double root ${\bar \l_1}$, so there is only one tangent line $\Gamma ^1_{\bar \l_1}$ for which  $i(C,\Gamma ^1_{\bar \l_1},O) \geq 3$, namely $\Gamma ^1_{\bar \l_1}: y=-\frac{a_{11}}{2a_{02}}x$.  Then $O$ is a non-ordinary singularity.
\dd We perform a linear change of coordinates so that the tangent at the double point $O$ to the transformed curve, which we still call $C$, is $y=0$, i.e. $\bar \l_1 =0$; the equation of $C$ then looks like:
$$
C: f(x,y) = y^2 + a_{30}x^3 + a_{21}x^2y + a_{12}xy^2 + a_{03}y^3 + ...     \eqno{(*)}
$$
If we have  $i(C,\Gamma ^1_{\bar \l_1},O) = 3$, i.e. if $a_{30}\neq 0$, then \underline{$O$ is an ordinary cusp $A_2$}, and the analysis ends here. Otherwise we go to step 2.

\a {\it Step 2:} {\bf Analysis via conics
 $\Gamma^2_\Lambda$.}  We are assuming that $C$ is given by $(*)$ with $a_{30} = 0$, that is, $i(C,\Gamma ^1_{\bar \l_1},O) \geq 4$. Consider  the pencil of conics $$\Gamma ^2_{{\Lambda}}:\   y={\bar \l_1}x+\l_2 x^2, \quad {\Lambda}= (\bar \l_1,\l_2)$$ all tangent to $C$ at $O$; since $\bar \l_1=0$, we get
 $$\Gamma ^2_{{\Lambda}}:\   y=\l_2 x^2, \quad {\Lambda}= (0,\l_2)$$ 
 $$
R(C,\Gamma ^2_{\Lambda}) =  (a_{40} + a_{21}\l_2 + \l_2^2)x^4 + (a_{50} + a_{31}\l_2 + a_{12}\l_2^2)x^5 + ...
$$
The coefficient of $x^4$ is zero when $$a_{40} + a_{21}\l_2 + \l_2^2=0 \eqno{(\star\star)}$$
There are two cases:   

\dd {\it Step 2a:} If the discriminant $a_{21}^2 - 4a_{40} \neq 0$, then  there are exactly two osculating conics  $\Gamma ^2_{ \Lambda 1}$, $\Gamma ^2_{\Lambda 2}$ for which   $i(C,\Gamma ^2_{\Lambda j},O) \geq 5$ ($j=1,2$), so \underline{$O$ is a double point $A_3$} (a tacnode), since the two conics separate the two branches of $C$ at $O$. And the analysis stops here. 

\dd {\it Step 2b:} If $a_{21}^2 - 4a_{40} = 0$, $(\star\star)$ has a unique root $\bar \l_2=-\frac{a_{21}}{2}$, so there is only one osculating conic  $\Gamma ^2_{\bar \Lambda}$ with $i(C,\Gamma ^2_{\bar  \Lambda},O) \geq 5$, the one with $\bar \Lambda=(0, \bar \l_2)$.
\b If  $i(C,\Gamma ^2_{\bar  \Lambda},O)= 5$, i.e. $a_{50} + a_{31}\bar \l_2 + a_{12}\bar \l_2^2\neq 0$, then \underline{$O$ is a cusp $A_4$}, and the analysis ends here. Otherwise we go to step 3.

\a {\it Step 3:} {\bf Analysis via cubics $\Gamma^3_\Lambda$.} We are assuming that $C$ is given by $(*)$ with $a_{30} = 0$, $a_{21}^2 - 4a_{40} = 0$, $2\bar \l_2 +a_{21}=0$ and $a_{50} + a_{31}\bar \l_2 + a_{12}\bar \l_2^2 =0$, that is,  $i(C,\Gamma ^2_{\bar \l_2},O) \geq 6$. Consider the pencil of cubics $$\Gamma ^3_{\Lambda}: y= \bar \l_2x^2 + \l_3x^3, \quad  \Lambda = (0,\bar \l_2,\l_3) $$
all having in common the tangent at $O$ (i.e. the tangent $y=0$ to $C$ at $O$), and the osculating conic at $O$ (i.e. the osculating conic $y=\bar \l_2x^2$ to $C$ at $O$). We get:
 $$
R(C,\Gamma ^3_\Lambda) =  (a_{40} + a_{21}\bar \l_2 + \bar \l_2^2)x^4 + (a_{50} + a_{31}\bar \l_2 + a_{12}\bar \l_2^2 + 2\bar \l_2\l_3 + a_{21}\l_3)x^5 + ...
$$
since the coefficients of $x^4$ and $x^5$ are $0$, we have $i(C,\Gamma ^3_\Lambda,O) \geq 6$.
\b The coefficient of $x^6$ is zero when
$$
\l_3^2 +(2a_{12}\bar \l_2+a_{31})\l_3 + a_{03}\bar \l_2^3 + a_{22}\bar \l_2^2 + a_{41}\bar \l_2+a_{60} =0  \eqno{(\star \star \star)}
$$
We again have two cases:

\dd {\it Step 3a:} If the discriminant $(2a_{12}\bar \l_2+a_{31})^2 - 4(a_{03}\bar \l_2^3 + a_{22}\bar \l_2^2 + a_{41}\bar \l_2+a_{60})\neq 0$,
there are exactly two osculating cubics  $\Gamma ^3_{ \Lambda 1}$, $\Gamma ^3_{\Lambda 2}$  for which $i(C,\Gamma ^3_{\Lambda j},O) \geq 7$ ($j=1,2$),  and \underline{ $O$ is a singularity $A_5$} (an oscnode).

\dd {\it Step 3b:} If the discriminant $(2a_{12}\bar \l_2+a_{31})^2 - 4(a_{03}\bar \l_2^3 + a_{22}\bar \l_2^2 + a_{41}\bar \l_2+a_{06})= 0$, $(\star\star\star)$ has a unique root $\bar \l_3$, so there is only one osculating cubic  $\Gamma ^3_{\tilde \Lambda}$ with $i(C,\Gamma ^3_{\tilde  \Lambda},O) \geq 7$, the one with $\tilde \Lambda=(0, \bar \l_2, \bar \l_3)$

\b If $i(C,\Gamma ^3_{\tilde \Lambda},O)= 7$, then \underline{$O$ is an $A_6$ cusp} for $C$. Otherwise we go to Step 4 where we use quartics $$\Gamma ^3_{\Lambda}: y= \bar \l_2x^2 + \bar \l_3x^3+ \l_4x^4, \quad  \Lambda = (0,\bar \l_2,\bar \l_3,\l_4)$$
 and we go on in the same way. This process will end at some point (see the justification of the Algorithm: if $O$ is an $A_{2r-1}$ or an $A_{2r}$ we will stop at Step r) .

\a  The procedure above is described in the following Algorithm \ref{Algorithm1True}. Notice that in the exposition above when the point is not a node we imposed $a_{02}\neq 0, a_{20}=a_{11}=0$ to have that the double tangent at $O$ is $y=0$ so to simplify computations.  This is not necessary, hence  in the algorithm we just impose $a_{02}\neq 0$.

\begin{algorithm}[H]\caption{Study of the double points of a plane curve $C: \sum a_{ij}x_0^ix_1^jx_2^{n-i-j}=0$}\label{Algorithm1True}
\a  \textbf{Input}: $F =  \sum a_{ij}x_0^ix_1^jx_2^{n-i-j} \in \mathbb{C}[x_0,x_1,x_2]_n$, $n>0$, $P=[a,b,c]$, $F(P)=0$.\\
 \a \textbf{Output}: State if $P$ is a double point for $C$ and its type: $A_m$.

\a \begin{boxedminipage}{125mm}

    \begin{algorithmic}[1]

	\STATE\label{Alg1Step0} {\it STEP 0)} \  Perform a linear change of coordinates so to have $P = [0,0,1]$; work in the affine chart $\{x_2\neq 0\}$, with $f =  \sum a_{ij}x^iy^j$, and $a_{00}=0$.
	
	   \dd If $(a_{01},a_{10})\neq(0,0)$: $P$ {\bf is a simple point for} $C$, with tangent $a_{01}x+a_{10}y=0$. {\bf STOP}
	
	   \dd If $(a_{01},a_{10})=(0,0)$ and $(a_{20},a_{11},a_{02})=(0,0,0)$: $P$ {\bf is a point of multiplicity $\geq 3$ for} $C$. {\bf STOP}. 
	   
	   \dd If $(a_{01},a_{10})=(0,0)$ and $(a_{20},a_{11},a_{02})\neq (0,0,0)$: $P$ {\bf is a double point for} $C$: {\bf go to Step 1}.
	
	  \a  \STATE\label{Alg1Step1}{\it STEP 1)} Perform a linear change of coordinates so to have $a_{02}\neq 0$.
	 Set $\Lambda := (\l_1)$ and consider $$R(C,\Gamma^1_{\Lambda}) :=f(x, \l_1x)= (a_{20}+a_{11}\l_1+a_{02}\l_1^2)x^2 + \dots$$

	 \STATE\label{Alg1Step1a}{\it STEP 1-a)} \ If  $\Delta^2_1(\Lambda) := a_{11}^2-4a_{20}a_{02}\neq 0$, then there exist $\l_{11}\neq \l_{12}$ with $i(C,\Gamma^1_{(\l_{11})},P)\geq 3$, $i(C,\Gamma^1_{(\l_{12})},P)\geq 3$ and $P$ {\bf is a double point for $C$ of type $A_1$} (ordinary node).\ {\bf STOP}.
	
	\a	\STATE\label{Alg1Step1b}{\it STEP 1-b)}\ If  $\Delta^2_1(\Lambda) = a_{11}^2-4a_{20}a_{02}= 0$, then there is a unique $\bar{\l}_1$ with $i(C,\Gamma^1_{(\bar{\l}_1)},P)\geq 3$.
	\dd {\it Step 1-$b_1$)} If $i(C,\Gamma^1_{(\bar{\l}_1),P)}= 3$, then $P$ {\bf is a double point for $C$ of type $A_2$} (ordinary cusp).\ {\bf STOP}.    
	\dd {\it Step 1-$b_2$)} If $i(C,\Gamma^1_{(\bar{\l}_1)},P) \geq 4$: {\bf go to Step 2}.
		
\medskip	
For $r\geq 2$, let 	
\a \STATE\label{Alg1Stepr}{\it STEP r)} \ Let $\Lambda = (\bar{\l}_1,\bar{\l}_2,...,\bar{\l}_{r-1}, \l_r)$ and
   $$\Gamma^r_{ \Lambda} : y=\bar{\l}_1 x +\bar{\l}_2x^2+...+\bar{\l}_{r-1}x^{r-1}+\l_r x^r $$ where the values $\bar{\l}_1,\dots,\bar{\l}_{r-1}$ come from Steps {\it 1-$b_2$,\dots, (r-1)-$b_2$}. Let $$R(C,\Gamma^r_{\Lambda}) :=f(x,\; \bar{\l}_1x+\bar{\l}_2x^2+...+\bar{\l}_{r-1}x^{r-1}+\l_r x^r )$$
 We have $i(C,\Gamma^r_{{\Lambda}},P)\geq 2r$, so the least power of $x$ appearing in $R(C,\Gamma^r_{\Lambda})$ is $2r$. Let  $\Delta^{2r}_r$ be the discriminant of the second degree equation in $\l_r$ obtained by forcing the coefficient of $x^{2r}$ to be zero. Then

 	\a \STATE\label{Alg1Stepr.a}{\it STEP r-a)} \  If $\Delta^{2r}_r\neq 0$, there exist $\l_{r1}\neq \l_{r2}$ such that, setting $\Lambda j = (\bar{\l}_1,\bar{\l}_2,...,\bar{\l}_{r-1}, \l_{rj})$, one has   $i(C,\Gamma^r_{ \Lambda j },P)\geq 2r+1$ for $j=1,2$, so $P$ {\bf is a double point for $C$ of type $A_{2r-1}$}.\ {\bf STOP}.

 	\a \STATE\label{Alg1Stepr.b}{\it STEP r-b)} \  If $\Delta^{2r}_r= 0$ there is a unique $\bar{\l}_r$  such that, setting $\bar \Lambda = (\bar{\l}_1,\bar{\l}_2,...,\bar{\l}_{r-1}, \bar \l_r)$, one has $i(C,\Gamma^r_{\bar \Lambda},P)\geq 2r+1$. 
	\dd {\it Step 1-$b_1$)} If $i(C,\Gamma^r_{\overline{\Lambda}},P)= 2r+1$  then $P$ {\bf is a double point for $C$ of type $A_{2r}$}.\ {\bf STOP}. 
	\dd {\it Step 1-$b_2$)} If $i(C,\Gamma^r_{\overline{\Lambda}},P)\geq 2r+2$: {\bf go to Step $r+1$}.
 	
    \end{algorithmic}
\end{boxedminipage}

\end{algorithm}

\bigskip
\begin{rem}  In the algorithm, the curves $\Gamma^r_{\overline{\Lambda}}$ need not to have degree $r$: some of the $\l_i$'s can be zero, or even all of them. For example, if $f(x,y) = y^2-x^5$, then $\Gamma^1_{(\bar \l_1)} = \Gamma^2_{(\bar\l_1,\bar\l_2)} = \{y=0\}$; since $i(C,\Gamma^2_{(\bar\l_1,\bar\l_2)},P)= 5$ this gives a verdict of $A_4$ singularity (actually this is the prototype of an $A_4$).
\end{rem}

\bigskip
\begin{rem}  The curve $C$ is assumed to be reduced, but it can be reducible, so it can happen that $i(C,\Gamma^r_{\bar \Lambda},P)= \infty$, when $\Gamma^r_{\bar \Lambda}$ is a component of $C$. 
For example, let $f(x,y) = y^2-yx^2 = y(y-x^2)$.  The algorithm gives that $\Gamma^1_{(\bar \l_1)}$ is given for $\bar \l_1=0$, and $i(C,\Gamma^1_{(\bar \l_1)},P)= \infty$; then in the next step, looking for $\Gamma^2_{\bar \Lambda}$, we find two possibilities: $\Gamma^2_{(0,0)}$ and $\Gamma^2_{(0,1)}$, both yielding $i(C,\Gamma^2_{\bar \Lambda},P)= \infty$, giving a verdict of $A_3$ singularity (tacnode) for $P$.
\end{rem}

\a {\bf Justification of Algorithm 1 procedure }

\bigskip
The justification relies mainly on Proposition 2.2, as we will see. 

\begin{prop}\label{Justification}
In the Hypotheses of Algorithm \ref{Algorithm1True} the following hold:

\noindent Algorithm \ref{Algorithm1True} stops at Step r-a if and only if P is an $A_{2r-1}$ singularity for $\mathcal{C}$;

\noindent Algorithm \ref{Algorithm1True} stops at Step r-b if and only if P is an $A_{2r}$ singularity for $\mathcal{C}$.
\end{prop}

\begin{proof} If $P$ is  is an $A_{s}$ singularity, $s\in \{2r-1,2r\}$, the statement follows directly by Proposition 2.2, {\it iv)}. 

\dd  If the algorithm stops at the step $1-a$,  $\l_{11}, \l_{12}$ give the two distinct tangents of an $A_1$ singularity (ordinary node).  If the algorithm stops at the step $1-b$, $\bar \l_1$ gives the unique tangent $\Gamma^1_{\bar \l_1}$ with $i(\mathcal{C},\Gamma^1_{\bar \l_1},P)=3$ of an ordinary cusp $A_2$.

\dd So now suppose $r\geq 2$.

\dd If the algorithm stops at the step $r-a$ the singularity cannot be an $A_{2h-1}$ with $h< r$, because in that case, by Proposition 2.2 {\it iv)}, the Algorithm should ha ve given two different $\Lambda1$, $\Lambda2$ at step $h-a$, nor it can be an $A_{2h}$ since in that case we should have $i(C,\Gamma^r_{\bar \Lambda},O) = 2h+1<2r+1$ .

\b On the other hand, the singularity cannot be an $A_s$, $s\geq 2r$, otherwise we should have  $\Gamma^r_{\Lambda_1} =\Gamma^r_{\Lambda_2} $, by  Proposition 2.2, {\it iv)}. Hence $P$ is an $A_{2r-1}$ singularity. 

\a If the algorithm stops at the step $r-b$, the singularity cannot be an $A_{2h}$ with $h<r$, because in that case we should have $i(C, \Gamma^r_{\overline \Lambda},P)\leq 2h+1<2r+1$, nor it can be an $A_{2h-1}$, $h\leq r$, since at step $h-a$ we should have got two different curves $\Gamma^h_{\Lambda 1}$,$\Gamma^h_{\Lambda 2}$.

\b On the other hand, the singularity cannot be an $A_s$, $s > 2r$, otherwise we should have that   $i(C, \Gamma^r_{\overline \Lambda},P)$ is even by Proposition 2.2 {\it i)}, and not $2r+1$. Hence $P$ is an $A_{2r}$ singularity.

\end{proof}

\section{An example of the use of Algorithm 1}

\a Let  $C_n \subset \PP ^n$ be a rational normal curve projecting on a plane curve $\pi: C_n \to C$ and let
 $O^r_Q(C_n)$ denote the $r$-dimensional osculating space to $C_n$ at the point $Q\in C_n$; recall that $O^r_Q(C_n)$ is the linear span in $\PP^n$ of the curvilinear scheme $(r+1)Q \subset C_n$. In \cite{BGI} we stated the lemma below without any proof, since we considered it as ``common knowledge":

\a {\rm (}\cite{BGI}, {\bf Lemma 4.2}{\rm )}.   {\it Let $P\in C\subset \PP^2$ be a double point on a rational curve of degree $n$, then:
 \begin{enumerate}
\item \label{nodo}The point $P$ is an $A_{2m-1}$, $2m-1<n$ if and only if the scheme $\pi^{-1}(P)$ is given by two distinct points $Q_1,Q_2\in C_n$  and $m$ is the maximum value for which  $\pi(O^{m-1}_{Q_1}(C_n))=\pi(O^{m-1}_{Q_2}(C_n))\neq \PP^2$.

\item\label{cuspfacile} The point  $P$ is an $A_{2m}$, $2m-1 < n$, if and only if the scheme $\pi^{-1}(P)$ is given by the divisor  $2Q\in C_n$  and $m$ is the maximum value for which  $\pi(O^{2m-1}_{Q}(C_n))\neq \PP^2$.
\end{enumerate}
\noindent When $m\geq 2$, we will have that the image $\pi(O^{m-1}_{Q_1}(C_n))=\pi(O^{m-1}_{Q_2}(C_n))$, (respectively $\pi(O^{2m-1}_{Q}(C_n))$) is $T_P(C)$, the unique tangent line to $C$ at $P$.}

\a This lemma, which luckily enough is never used in \cite{BGI}, is actually wrong, but it sounded quite convincing not only for us, since neither the referee of \cite{BGI} (which otherwise made a quite thorough job), nor several colleagues with whom we talked about it during the making of the paper realized that it does not hold.

\a Here is the rationale which explains while the Lemma does not hold: assume $m>1$ and that we are in case $(1)$ (but analogous considerations may be done in case $(2)$); consider the linear spans, giving the osculating spaces: $\langle mQ_1  \rangle = O^{m-1}_{Q_1}(C_n), \langle mQ_2  \rangle  = O^{m-1}_{Q_2}(C_n)$. The Lemma states that the projection of these two osculating spaces to the curve $C_n$ is not the whole of $\mathbb{P}^2$, so it has to be the unique tangent line $t$ to $C$ at $P$; but if this were true, the projection $X$ of the two curvilinear schemes $mQ_1$, $mQ_2$ would  be contained in a line (the tangent $t$), and this is not true in general, as the next example, which makes use of Algorithm 1 of section 5, shows.

\begin{ex}\label{controes}
Consider the quartic curve $C $ given by the affine equation $y^2-2x^2y+x^4+x^2y^2=0$. If we run Algorithm 1 on it, we find a unique tangent $y=0$ in Step 1, a unique osculating conic $\Gamma: y=x^2$ in Step 2, and two distinct osculating cubics in Step 3: 
$$D_1: y=x^2-ix^3,\quad D_2: y=x^2+ix^3$$
Hence Algorithm 1 gives that $O$  is an oscnode for $C$ (an $A_5$ double point, so here $m=3$), and $C$, being a quartic with an oscnode, is rational, with two branches at $O$  approximated by $D_1$ and $D_2$. Let's denote by $C_4\subset \PP ^4$ a rational normal curve which projets onto $C$.

We have $i(C,D_j,O)=7$ for $j=1,2$ and $i(D_1,D_2,O)=3$, in accord with Proposition \ref{Lemma} $(ii)$, and the length 3 curvilinear scheme $X$ of Theorem 4.3 in \cite{BGI}, or of Theorem \ref{2mtom notrat} where we take $\tilde C=C_4$, is given by $D_1\cap D_2$, hence it is associated to the ideal $(y-x^2-ix^3,y-x^2+ix^3)=(y-x^2,x^3)$.
Hence $X$  is contained in the osculating conic $\Gamma$ and not contained in the tangent line; but, according to Theorem \ref{2mtom notrat}, $X$ is the projections of the curvilinear length 3 schemes $3Q_1$ and $3Q_2$ of  $C_4\subset \PP ^4$; so we conclude that $3Q_1$ and $3Q_2$ are not projected inside the tangent $y=0$.

\end{ex}

\section{Singular points on a rational plane curve via parameterization}\label{Xk}

\b In \cite{BGI} we expose a way to determine the nature of a singularity on a rational plane curve $C$, given a parameterization of $C$,  without using its implicit equation or the syzygies of the parameterization (e.g. as in \cite{CKPU}); in this section we recall a few result from \cite{BGI}.

\begin{defn}\label{defXk} Let $C\subset \mathbb{P}^2$ be a rational curve of degree $n\geq 3$, given by a map  ${\mathbf{f}}= (f_0,f_1,f_2): \PP ^1 \to C$, and assume that the parameterization $(f_0, f_1, f_2)$ is proper, i.e. ${\mathbf{f}}$ is generically 1:1 and the $f_i$'s do not have common zeroes. Let $$f_j = a_{j0}s^n + a_{j1}s^{n-1}t + \cdots + a_{jn}t^n, \quad j=0,1,2$$
Consider the $(n-k+4)\times (n+1)$ matrices:

$$
M_k=  \begin{pmatrix} x_0&x_1&\ldots &x_k&0&0&\cdots &0 \cr
0&x_0&x_1&\cdots &x_k&0&\cdots&0
\cr  & &\ddots&&  & & &  \cr
0&\cdots&0&x_0&x_1&\cdots &\cdots &x_k\cr
 a_{00}& a_{01}& a_{02}& a_{03}& \cdots& \cdots & a_{0n-1} &a_{0n}\cr
a_{10}& a_{11}& a_{12}& a_{13}& \cdots & \cdots & a_{1n-1} &a_{1n}\cr a_{20}& a_{21}& a_{22}& a_{23}& \cdots & \cdots& a_{2n-1} &a_{2n}
\end{pmatrix}.
$$
\dd For $2\leq k \leq n-1$, we denote by  $X_k\subset \PP^k$  the scheme defined by the $(n-k+3)$-minors of $M_k$ .
\end{defn}

\a The following Proposition (see \cite{BGI}, Prop. 2.2) shows how the $X_k$'s are related to the $k$-uple points of $C$:

\begin{prop}\label{Singpoint}
Let $C\subset \PP^2$ be a rational curve of degree $n\geq 3$. The schemes $X_k$ introduced in Definition \ref{defXk} are either 0-dimensional or empty. Moreover:
\begin{itemize}
\item $\forall \, k$, $2\leq k\leq n-1$, $X_k$ is non-empty iff there is at least a singular point on $C$ of multiplicity $\geq k$.

\item  Every singular point of $C$ yields at least a simple point of $X_2$ and $$length X_2 =  {n-1\choose 2}$$ (notice that $X_2$ is never empty since $n\geq 3$).
\end{itemize}
\end{prop}

\b There are several properties of the singularities of $C$ which are quite immediate to check using the schemes $X_k$ (see \cite{BGI}, Prop. 3.1 and Prop. 4.4); we report some of them in the following proposition \ref{Singpoint2}, where, if $P\in \Sing(C)$, $\delta_P$ denotes the number 
$\delta_P = \sum_q {m_q\choose 2}$  where $q$ runs over all points infinitely near $P$, and $\mathcal{C}_2$ is the conic $y^2-4xz=0$. 
\begin{prop}\label{Singpoint2}
 Let $C\subset \PP^2$ be a rational curve, given by a proper parameterization $(f_0,f_1,f_2)$, with $f_i\in K[s,t]_n$. Let  $\mathcal{C}_2, X_2$ be as defined before. Then:

\begin{itemize}
\item $C$ is cuspidal if and only if $\Supp(X_2) \subset \mathcal{C}_2$ (in this case, the number of singular points of $C$ is exactly the cardinality of $\Supp(X_2)$).

\item $C$ has only ordinary singularities if and only if the scheme $X_2$ is reduced and $X_2\cap \mathcal{C}_2 = \emptyset$.

\item  Let $C$ have only double points as singularities and let $R\in X_2$ and $P\in \Sing(C)$ be the point associated to $R$. Then $length_R(X_2) = \delta_P$.
\end{itemize}

\end{prop}

\a The algorithms given in \cite{BGI} describe how to use the scheme $X_2$ in order to study double points; in the next section we show an example of how to construct a desired curve with a double point of type $A_m$.

\a 
\section{An example of the use of techniques of section \ref{Xk}}

 \a In this section, with \cite{CoCoA}, we build an example  which illustrates how to build a plane rational curve with a double point of a chosen type using projection techniques, and how to use the results of section \ref{Xk} to study the singularities of a rational curve, given its parameterization. This example also gives another counterexample to  Lemma 4.2 in  \cite{BGI}.

\begin{ex}\label{excocoa}

In the first part of this example we construct a rational sextic   $C\subset \PP^2$  with an $A_5$ singularity  $P$ (an $oscnode$), and we show that it is a counterexample to  \cite{BGI} 4.2. In the second part we construct the scheme $X_2$ relative to our curve and we complete the study of the singular locus of the curve.
\dd {\bf Part I} We use Theorem \ref{2mtom notrat} as a guide; hence, in order to obtain an $A_5$ singularity, we want to view our curve as the projection of a rational normal curve $C_6\subset \PP^6$, with center a linear space $\Pi \cong \PP^3$, onto a plane $H\subset \PP^6$, in such a way that two points $Q_1$ and $Q_2$ on $C_6$  have the same image $P$, and moreover the two curvilinear schemes $3Q_1$, $3Q_2$ on $C_6$ have a curvilinear scheme $X \subset C$ of length three as their projection, with $X$ not contained in a line.
\dd The idea is the following: let $Q_1 = [0,0,0,0,0,0,1]$, $Q_2 = [1,0,0,0,0,0,0]$, and let $L$  be the line through them. The ideal $I_L^3 + I_{C_6}$ in $K[z_0,...,z_6]$ defines the required curvilinear scheme $3Q_1 + 3Q_2$. We want to project $C_6$ from a 3-dimensional space $\Pi$ into a plane $H$, choosing $\Pi$ in such a way that the projection $\pi : C_6 \rightarrow C$ is generically 1:1; $\Pi$ does not intersect $C_6$ and  intersects $L$ at one point, so that $\pi(Q_1)=\pi(Q_2)$; it does not intersect  the two osculating spaces  $O_{Q_i}^2(C_6)$, so that $\pi(O_{Q_1}^2(C_6))=\pi(O_{Q_2}^2(C_6))=H$; the image of $3Q_1+3Q_2$ is a curvilinear scheme of length 3 contained in $C$; $\pi(4Q_1+4Q_2)$ is not a degree 4 curvilinear scheme on $C$.  

\dd If we manage to do so, the curve $C =\pi(C_6)\subset H$  will have an $A_5$ singularity in  $P=\pi(Q_i)$ by Theorem \ref{2mtom notrat}. Moreover, the image $\pi(O_{Q_i}^2(C_6))$, $i=1,2$, will not be contained in a line, contradicting Lemma 4.2 of \cite{BGI}. 

\a In the following we describe the procedure by using the program CoCoA (see \cite{CoCoA}):

\bigskip Use R::= QQ[$a,b,c,d,e,f,g$];

This is the ring of coordinates of $\PP^6$.

\medskip

IL:= Ideal$(b,c,d,e,f)^3$;

\noindent This is the ideal of the ``triple line" $L$ through the points $A=[1,0,0,0,0,0,0]$ and $ B=[0,0,0,0,0,0,1]$  in $\PP^6$.

\medskip

IC6 = Ideal($ac-b^2,ad-bc,ae-bd,bd-c^2,be-cd,ce-d^2,df-e^2,de-cf,ce-bf,af-be, ag-bf,bg-cf,cg-df,dg-ef,eg-f^2$);

IP:= IL+ IC6;

\noindent This is the ideal $I_L^3+I_{C_6}$ of the curvilinear scheme   $3A+3B$, supported on  $C_6$.

\medskip

\noindent Now we consider the space  $\Pi \cong \PP^3$ whose ideal is $ (a+g, 3f-b-d, 9e+c-d)$; $\Pi$ intersects $L$ in a point and does not intersect the two osculating planes $ O^2_A(C_6)$ (whose ideal is $(d,e,f,g)$) and $O^2_B(C_6)$ (whose ideal is $(a,b,c,d)$). We want to project with center $\Pi$ on the plane with coordinates $u,v,w$, where $u=a+g,v=3f-b-d,w=9e+c-d$.

\medskip

Use R::= QQ[$a,b,c,d,e,f,g,u,v,w$];

IS:= IP+Ideal($u-a-g,v-3f+b+d,w-9e-c+d$);

IIS:=Saturation(IS,Ideal($a,b,c,d,e,f,g$));

Elim($a..g$,IIS);

Ideal($w^2, vw, v^2 - uw$)

\noindent The projection of $3A+3B$ is a scheme whose ideal $(w^2, vw, v^2 - uw)$ shows that it is supported at  $P=[1,0,0]$, it has length 3 (it is generated by 3 independent conics), it is curvilinear (it is contained in a smooth conic) and is not on a line, hence  we have a good candidate for an $A_5$ at $P$.

\medskip

\noindent Let us check that $P$ is not an $A_7$, we will go through the same steps starting with the ideal $I_L^4$:

IL4:= Ideal$(b,c,d,e,f)^4$;

IP4:= IL+IC6;

IS4:= IP+Ideal(u-a-g,v-3f+b+d,w-9e-c+d);

IIS4:=Saturation(IS,Ideal(a,b,c,d,e,f,g));

Elim(a..g,IIS4);

Ideal($w^2, 9/28v^2w, 9/28v^3 - 9/28uvw$)

\medskip

\noindent The ideal we got is not the ideal of a curvilinear scheme, since the three curves defined by  $w^2, 9/28v^2w$ and $9/28v^3 - 9/28uvw$ are not smooth at $P$, and it can be checked that it has lenght 5. Hence $P$ is not an $A_7$ by Theorem \ref{2mtom notrat}.

\medskip

We are left to check that the projection is generically 1:1; for this it is enough to find a smooth point of the plane curve $C$ which comes from only one point of $C_6$ via $\pi$.  Let us consider the point $R=[1,1,1,1,1,1,1]\in C_6$:

Use R::= QQ[$a,b,c,d,e,f,g,u,v,w$];

IR:= Ideal($a-b,b-c,c-d,d-e,e-f,f-g$);

IR1:= I+ Ideal($u-a-g,v-3f+b+d,w-9e-c+d$);

IIR1:=Saturation(I1,Ideal($a,b,c,d,e,f,g$));

Elim(a..g,II1);

Ideal($v - 1/9w, u - 2/9w$)

\noindent This shows that $R$ projects to the point $[1,2,9]\in C$. Now we consider the (4-dimensional) cone on $\Pi$ with vertex $[1,2,9]$ and we intersect it with $C_6$.

Use R::= QQ[$a,b,c,d,e,f,g$];

ICONO:= Ideal($9e+c-d-27f+9b+9d,18e+2c-2d -9a-9g$);

IC6 := Ideal($ac-b^2,ad-bc,ae-bd,bd-c^2,be-cd,ce-d^2,df-e^2,de-cf,ce-bf,af-be, ag-bf,bg-cf,cg-df,dg-ef,eg-f^2$);

I:= IC6+ICONO;

II:=Saturation(IP,Ideal($a,b,c,d,e,f,g$));

Print II;

Ideal($a - g, b - g, c - g, d - g, e - g, f - g$)

\noindent Hence the cone intersects $C_6$ only in $R$ (simply), and so $\pi$ is generically 1:1.

\a {\bf Part II}  The singularity of $C$ at $P$ is now known; let us complete the study of the curve $C$ by checking what the other singularities are.
Since the ideal of $\Pi $ is $ (a+g, 3f-b-d, 9e+c-d)$, the parametric equations  of $C$ are: 
$$u= s^6 + t^6\quad  ;\  v= -s^5t+3st^5-s^3t^3\quad ;\  w = 9s^2t^4 + s^4t^2-s^3t^3$$

\noindent We want to study the scheme $X_2 \subset \PP^2$, defined by the $7\times 7$ minors of $M_2$:

 \dd Use R::= QQ[$x,y,z$];

M:=Mat($[[x,y,z,0,0,0,0],[0,x,y,z,0,0,0],[0,0,x,y,z,0,0],[0,0,0,x,y,z,0],$

$[0,0,0,0,x,y,z],[1,0,0,0,0,0,1],[0,-1,0,-1,0,3,0],[0,0,1,-1,9,0,0]]$);

MM:=Minors(7,M);

IX2:=Ideal(MM);

\noindent This is the ideal of the scheme $X_2$.

\dd Hilbert(R/IX2);

H(0) = 1

H(1) = 3

H(2) = 6

H(t) = 10   for t $\geq$ 3

\noindent $ X_2 $ has lenght 10, as expected ($C$ is a rational sextic).

\dd IZ2:=Radical(IX2);

Hilbert(R/IZ2);

H(0) = 1

H(1) = 3

H(2) = 6

H(t) = 8   for t $\geq$ 3

\dd $X_2$ has support at $Z_2$ which is made of 8 points, hence Sing$C$ is made of  our $A_5$ supported on $P$ plus 7 double points of type $A_1$ or $A_2$; to decide if they are nodes or cusps, we proceed as follows:

\a ICUSP:= IX2+Ideal($y^2-4xz$);

Hilbert(R/ICUSP);

H(0) = 1

H(1) = 3

H(2) = 5

H(3) = 7

H(4) = 4

H(t) = 0   for t $\geq$ 5

\a  $X_2$ does not intersect the conic which is the locus of points parameterizing tangent lines of $C_6$, hence the seven simple points are all ordinary nodes $A_1$.

\end{ex}


\end{document}